\documentclass[11pt,letterpaper,oneside]{amsart}

  \usepackage{amsmath,amssymb,amsthm}
  \usepackage{geometry}
  \usepackage{accents}
  \usepackage{graphicx}
  \usepackage{enumerate}
  \usepackage{color}
  \usepackage{xspace,verbatim}
 \usepackage{hyperref}
\def\R{\mathbb{R}}
 \def\Z{\mathbb Z}

\newtheorem{theorem}{Theorem}[section]
\newtheorem{corollary}[theorem]{Corollary}

\newtheorem{proposition}[theorem]{Proposition}
\newtheorem{lemma}[theorem]{Lemma}
\newtheorem{remark}[theorem]{Remark}

\newtheorem{question}[theorem]{Question}
\newtheorem{conjecture}[theorem]{Conjecture}

\title{Small dilatation pseudo-Anosov mapping classes and short circuits on train track automata}
\author{Eriko Hironaka}

\begin{document}

\begin{abstract}
This note is a survey of recent results
 surrounding the minimum dilatation problem for pseudo-Anosov mapping classes.
In particular, we give evidence for the conjecture that the minimum accumulation point of 
the genus normalized dilatations of pseudo-Anosov mapping classes on closed surfaces
equals the square of the golden ratio.  We also find explicit fat train track maps determining
a sequence of pseudo-Anosov mapping classes whose normalized dilatations
converge to this limit.  \end{abstract}

%\begin{abstract}Conjecturally, minimal dilatation mapping classes have dilatations that
%satisfy LT-polynomials.  In this paper, we 
%explicitly describe the dynamics of a family of  pseudo-Anosov
%mapping classes achieving these conjectural minima
%in terms of  train tracks, and show that they correspond
%to length 3 circuits on train track automata.
%\end{abstract}
\maketitle

\section{Introduction} 
Let $S$ be a compact surface of genus $g$ with $b$ boundary components.
 A {\it mapping class} $\phi$ on $S$ is
a self-homeomorphism of $S$ considered up to isotopy.  
The map $\phi : S \rightarrow S$ is {\it pseudo-Anosov}
if $S$  admits a pair of $\phi$-invariant transverse measured singular foliations, called the
 {\it unstable foliation} $(\mathcal F^u,\nu^u)$ and  {\it stable foliation} $(\mathcal F^s,\nu^s)$, so that
the action of $\phi$ stretches  $\nu^u$ by a constant $\lambda > 1$, and contracts 
$\nu^s$ by $\frac{1}{\lambda}$.  The constant $\lambda$ has the property that
 $\log (\lambda)$ is the minimal topological entropy of elements in the isotopy class of $\phi$
 and is called the {\it dilatation} of $\phi$.
 The theory of pseudo-Anosov mapping classes is developed in detail in \cite{FLP79}, \cite{CB88}
and \cite{Thurston88}.

In a 1991 paper, Penner  \cite{Penner91} proved
that as a function of genus $g \ge 2$,  the minimum dilatation $\delta_g$ for pseudo-Anosov mapping
classes on closed genus $g$ surfaces satisfies
\begin{eqnarray}\label{Penner-eqn}
\log \delta_g \asymp \frac{1}{g}.
\end{eqnarray}
Penner's paper has brought recent interest to
the minimum dilatation problem, which asks what are the values of
$\delta_g$ for $g \ge 2$, and what are the mapping classes that realize these values.
So far the exact value of the  minimum dilatation $\delta_g$ is known only for $g = 2$ \cite{CH08}.   
In this paper we give a brief survey of the minimum dilatation problem and its relations to the study of train track maps,
digraphs, polynomials and algebraic integers, and give an illustrative example.

\subsection{Lehmer's problem and dilatations}
Questions surrounding the values of $\delta_g$
are closely analogous to Lehmer's problem on Mahler measures. 
Dilatations of pseudo-Anosov mapping classes are special algebraic integers called Perron numbers.  These
are real algebraic integers $\lambda > 1$ all of whose algebraic conjugates are strictly smaller
in complex norm.  Furthermore, dilatations have the property that $\lambda^{-1}$ is also
an algebraic integer, and hence $\lambda$ is an algebraic unit.  The Mahler measure 
$m(\lambda)$ of an algebraic integer 
$\lambda$ is the absolute value of the product of its conjugates outside the unit circle.  
In \cite{Lehmer33} Lehmer asks: is there is a positive gap  between 1 and the
next largest Mahler measure?  A negative answer would mean that the set of Mahler measures is dense in the interval $[1,\infty)$.
Lehmer's question leads immediately to several others.

For each fixed degree $n$, any bound on Mahler measure bounds the size of the coefficients
of the minimal polynomial, and hence the Mahler measures greater than one for 
algebraic integers of fixed degree $n$ achieve a minimum
$m_n > 1$.  It is not known how $m_n$ behaves as $n$ goes to infinity, 
nor about properties of the algebraic integers achieving $m_n$.  For example: is there
a bound on the number of algebraic conjugates outside the unit circle?

The 
complex norm $h(\lambda)$ of the largest conjugate of an algebraic integer $\lambda$ is called  the {\it house}
of $\lambda$.  The  {\it normalized house} 
$$h(\lambda)^{d_{\mbox{\tiny alg}}}$$ 
is the house raised to the degree of the minimal polynomial.  
It is not known whether this coarse upper 
bound for Mahler measure is bounded away from one for non-cyclotomic algebraic integers (cf. \cite{Dobrowolski79}).

\subsection{Perron numbers} For Perron numbers, there is an alternative way to normalize house, other than algebraic degree.  Each Perron number is the spectral radius of a {\it Perron-Frobenius matrix}:
a $d \times d$ matrix $M$ with non-negative integer entries such that for some power $k \ge 1$, $M^k$ has strictly positive entries.  The
 minimum such $d$, which is an upper bound for $d_{\mbox{\tiny alg}}$, is the degree of the characteristic polynomial of $M$,
 called the {\it Perron-Frobenius degree} of the Perron number.  McMullen
 recently showed in \cite{McM:Clique}
  that for Perron units $\lambda$ with Perron-Frobenius degree $d_{\mbox{\tiny PF}}$, we have
 \begin{eqnarray}\label{McMullen-eqn}
 \lambda^{d_{\mbox{\tiny PF}}} \ge \gamma_0^4,
 \end{eqnarray}
 where $\gamma_0$ is the golden ratio.

\subsection{Normalized dilatations}
It is an open question whether all Perron units are dilatations of pseudo-Anosov mapping classes (partial results in this direction were found by
Thurston in \cite{Thurston:Perron}).
Define the   {\it genus-normalized} dilatation to be
$ \lambda(\phi)^g$
and let $\ell_g = (\delta_g)^g$, the minimum genus-normalized dilatation for fixed genus $g$.
Penner's result (\ref{Penner-eqn}) is equivalent to the statement that there are constants $c$ and $C$ so that
$$
1 < c \leq \ell_g \leq C.
$$
It is an open problem to determine sharp bounds for $c$ and $C$, or to find the limit of $\ell_g$
as $g$ goes to infinity.

McMullen's result (\ref{McMullen-eqn}) on normalized Perron units is evidence for the following conjecture.

\begin{conjecture}\label{goldenmean0-conj}  The smallest accumulation point for the sequence $\ell_g$
is $\gamma_0^2$.
\end{conjecture}

\noindent
For the pseudo-Anosov mapping classes 
$(S_g,\phi_g)$ that we later describe in this paper, the surfaces
$S_g$ have genus $g$, the normalized dilatations $\lambda(\phi_g)^g$ converge to $\gamma_0^2$,
hence $\gamma_0^2$ is an upper bound for the smallest accumulation point.  
This together with McMullen's result (\ref{McMullen-eqn}) 
is not enough to prove the conjecture, however, since in general
 both $d_{\mbox{\tiny alg}}$ and $2g$ can be strictly smaller than $d_{\mbox{\tiny PF}}$, and
the latter can be as large as $6g-6$ \cite{Penner91}.  

Conjecture~\ref{goldenmean0-conj} was  originally inspired by a question of
Lanneau and Thiffeault posed in \cite{LT09}.  An orientable pseudo-Anosov mapping class
is one where the stable and unstable foliations are orientable.   Lanneau and Thiffeault
ask whether for orientable 
pseudo-Anosov mapping classes on surfaces of even genus, the minimum dilatation is the largest real root of the polynomial
$$
LT_n(x) = x^{2n} - x^{n+1} - x^n - x^{n-1} + 1.
$$
If $\lambda_n$ is the largest root of $LT_n(x)$, then it is not hard to show that 
$(\lambda_n)^n$ is a monotone decreasing sequence converging to $\gamma_0^2$.

\subsection{Main example}
In this paper, we
explicitly define a sequence of pseudo-Anosov mapping classes whose genus normalized dilatations
define a strictly monotone decreasing sequence converging to $\gamma_0^2$.  The existence of such sequences was already proved in \cite{Hironaka:LT}  \cite{AD10} and \cite{KT11}, but the description
we give here, using the language of fat train track maps and digraphs, is the  first constructive one, 
and serves to give a glimpse of what small dilatation mapping classes  look like in general.

We show the following.

\begin{theorem}\label{example-thm} There is a sequence of pseudo-Anosov mapping classes $(S_n,\phi_n)$
described by fat train track maps $f_n : \tau_n \rightarrow \tau_n$, $n \ge 2$ with the following properties:
\begin{enumerate}
\item  $S_n$ is a closed orientable surface of genus $g=n$ if $3$ doesn't divide $n$ and
genus $g=n-1$ if $3$ divides $n$,
\item $\lambda(\phi_n)$ is the largest real root of $LT_n(x)$,
\item the genus-normalized dilatations of $(S_n,\phi_n)$ converge to $\gamma_0^2$.
\item $(S_n,\phi_n)$ is an orientable mapping class if and only if $n$ is
even,
\item $(S_n,\phi_n)$ have the smallest dilatation among orientable pseudo-Anosov mapping classes
of genus $g = n$ when $n = 2,4,8$, and of genus $g = 5$ when $n = 6$.
\item the train track maps $f_n$ have folding decompositions 
corresponding to  length 3 circuits on fat train track automata, and
\item the topological type of the digraph associated to the train track map $f_n$ is fixed for $n \ge 2$.
\end{enumerate}
\end{theorem}

\begin{corollary}\label{goldenmean-cor} The square of the golden mean $\gamma_0^2$ is an accumulation
point for normalized dilatations of orientable pseudo-Anosov mapping classes.
\end{corollary}

\noindent
Sequences satisfying properties (1)--(5) were also found
in \cite{Hironaka:LT} as mapping classes associated 
to a convergent sequence on a fibered face.  The difference in this paper
is that our description is constructive.

\subsection{Organization}
Thurston's fibered face theory \cite{Thurston:norm}, Fried's results about cross-sections of
   pseudo-Anosov flows \cite{Fried82}, 
McMullen's theory of Teichm\"uller polynomials \cite{McMullen:Poly} and the universal finiteness theorem of
Farb, Leininger and Margalit \cite{FLM09}  together imply that the
problem of finding minimum dilatations reduces to understanding the roots of  families of polynomials arising as specializations
of a finite list of multivariable polynomials.   We recall these results in Section~\ref{fibered-section}.  
In Section~\ref{Perron-section} we describe the restriction of Lehmer's problem to Perron units, and its
recent partial solution by McMullen \cite{McM:Clique}.
The special case of orientable pseudo-Anosov
 mapping classes, and the Lanneau-Thiffeault question 
 is discussed in Section~\ref{orientable-section}.
In Section~\ref{traintracks-section} we define 
fat train track maps, and their automata.  We also explain how to compute 
Both the Teichm\"uller and Alexander polynomials in this context.
In Section~\ref{example-section}, we describe a sequence of fat train track maps 
whose Teichm\"uller polynomial specializes to the LT polynomials, and prove Theorem~\ref{example-thm}.

\section{Fibered faces, dilatations and polynomials}\label{fibered-section}
Fibered face theory gives a natural way to partition the set of pseudo-Anosov mapping classes
into families that are in one-to-one correspondence with rational points on convex Euclidean
polyhedra (possibly single points).   Each family contains mapping classes
defined on different surfaces, but having related dynamics.   In particular, the normalized dilatation
varies continuously with respect to the induced Euclidean metric.  Furthermore, each set
has an associated Teichm\"uller  polynomial, whose specialization at each point in the set 
determines the dilatation of the associated mapping class.

\subsection{Fibered face theory}
In \cite{Thurston:norm}, Thurston defines a norm $|| \ ||$ on $H^1(M;\R)$ as follows.
Given a surface $(S,\partial S) \subset (M, \partial M)$, let 
$$
\chi_-(S) = \sum_{S' \subset S} \max\{-\chi(S'), 0\},
$$ 
where the sum is taken over connected components $S'$ of $S$.
Given $\alpha \in H^1(M;\Z)$, let 
$$
||\alpha|| = \min \{\chi_-(S) \  : \ \mbox{$(S,\partial S) \subset (M,\partial M)$ is Poincar\'e dual to $\alpha$}\}.
$$
Then $||\ ||$ extends to a unique norm on $H^1(M;\R)$.
Furthemore, the unit norm ball is a convex polyhedron, and the convex hull of rational vertices.
The norm $|| \ ||$ is called the {\it Thurston norm}, and the unit ball is called the {\it Thurston norm ball}.

An element of $H^1(M;\Z)$ is called {\it fibered} if it is dual to the fiber of  a fibration $\psi_\alpha : M \rightarrow S^1$
over the circle.

\begin{theorem}[Thurston \cite{Thurston:norm}]
For every open  top-dimensional face $F$ of the unit Thurston norm ball,  either every integral point
in the cone $F \cdot \R^+$ over $F$ is fibered, or none of them are.  
\end{theorem}

If the integral points on $F \cdot \R^+$ are fibered, we say $F$ is a {\it fibered face} and $F \cdot \R^+$ is a
{\it fibered cone}.

 Circle fibrations of $M$ are in one-to-one correspondence
with mapping classes $(S,\phi)$ with the property that $M$ is the mapping torus of $(S,\phi)$:
$$
M \simeq S \times [0,1]/{(x,1) \sim (\phi(x),0)},
$$
where $S$ is homeomorphic to the {\it fiber} of the fibration.
The mapping class $(S,\phi)$ is called the {\it monodromy} of the fibration.  

A {\it primitive integral element} in $H^1(M;\Z)$ is a point with
relatively prime integral coefficients.  
Given a fibered element $\alpha \in H^1(M;\Z)$, any positive integer multiple $m \alpha$ has the property that
$\psi_{m\alpha}$ is the composition of $\psi_\alpha$ with the $m$-fold cyclic covering of the circle to itself.
If follows that primitive integral elements on fibered cones correspond to fibrations of $M$ over the circle
with connected fibers. 

A key theorem of Thurston that connects the classification of mapping classes and that of fibered 3-manifolds
is the following.

\begin{theorem}[Thurston \cite{Thurston88}]\label{hyperbolicity-thm}
A mapping class is pseudo-Anosov if and only if its mapping torus
is a hyperbolic 3-manifold.  
\end{theorem}

\noindent
It follows that there is a one-to-one correspondence between pseudo-Anosov mapping classes $(S,\phi)$ on  surfaces $S$
and rational points on fibered faces of  hyperbolic 3-manifolds
whose denominator equals $|\chi(S)|$.  

\subsection{Removing singularities}
To study the dynamical properties of a pseudo-Anosov mapping class  it is natural to
remove the singularities of the invariant stable and unstable foliations.  This process
preserves essential information about the surface (e.g., genus) and the
dynamics of the mapping class (e.g., dilatation).   In many cases, this process increases
 the first Betti number of the mapping torus, and hence the dimension of the
 associated fibered face.

\begin{lemma}\label{Betti-lem}  Let $S$ be a compact surface with boundary,
and $\phi$ a pseudo-Anosov map on $S$.  The first Betti number of the mapping torus of $(S,\phi)$ is 
$r + 1$, where $r$ is  the rank of the 
$\phi$-invariant homology $H_1(S,\partial S;\Z)$.
\end{lemma}

\begin{proof} 
See, for example, \cite{McMullen:Poly}.
\end{proof}

Define the {\it singularities} of a pseudo-Anosov mapping class $(S,\phi)$ to be the set of singularities of the stable and unstable
$\phi$-invariant foliations. The union of singularities on $S$ is a finite set of
points closed under the action of $\phi$.  Let $S^0$ be 
the complement of small neighborhoods of the singular points.
There is a unique pseudo-Anosov mapping class $\phi^0$ defined on
$S^0$ determined up to isotopies that fix the boundary component pointwise.
Correspondingly, there is a well-defined way to define invariant
foliations for $\phi^0$ whose extensions to $S$ are the original invariant foliations
of $\phi$, so that certain leaves terminate at the boundary.  The leaves terminating at a 
boundary component are called prongs, and the degree of the singularity equals the number
of prongs minus 2.

By this construction, 
the dilatations $\lambda(\phi)$ and $\lambda(\phi^0)$ are stretching factors of the same maps
on the same foliations,
and hence are equal.
Furthermore,  $(S,\phi)$ can be recovered from $(S^0,\phi^0)$ by closing off the boundary 
components with disks.

\begin{corollary}  Suppose $(S,\phi)$ is a pseudo-Anosov mapping class such that
the number of orbits of boundary components and the number of orbits of singularities
add up to at least 2.  Then the first Betti number of the mapping torus of $(S^0,\phi^0)$
is greater than or equal to 2, and hence $(S^0,\phi^)$ corresponds to a point on
a fibered face of positive dimension.
\end{corollary}

\begin{proof}
For any mapping class $\phi$ on a surface with boundary
$S$,  the sum $\gamma$ of loops around the orbits of a boundary component
determines a $\phi^0$-invariant element $[\gamma]$ in $H_1(S^0, \partial S^0;\Z)$.
If there is more than one orbit, then $[\gamma]$ is non-trivial.
The rest follows from Lemma~\ref{Betti-lem}.
\end{proof}

\subsection{Normalized dilatations}
The {\it normalized dilatation} of a pseudo-Anosov mapping class $(S,\phi)$ is defined by
$$
L(S,\phi) = \lambda(\phi)^{|\chi(S)|}.
$$
Given a fibered element $\alpha \in H^1(M;\Z)$ with monodromy $(S_\alpha,\phi_\alpha)$ define
$$
\mathcal H(\alpha) = \log(\lambda(\phi_\alpha)).
$$
When $\alpha$ is an integral element,  $\mathcal H(\alpha)$  is  the {\it topological entropy} of $\phi_\alpha$.

\begin{theorem}[Fried \cite{Fried82}, McMullen \cite{McMullen:Poly}] \label{Fried-thm} The function $\mathcal H(\alpha)$ extends to a 
real analytic, convex function that is homogeneous of degree $-1$ on each fibered cone $F \cdot \R^+$ and goes to infinity toward the
 boundary of the fibered face $F$.
 \end{theorem}
 
Given a primitive integral point $\alpha \in F \cdot \R^+$, let $\overline \alpha = \alpha/q$ be its projection onto $F$.

\begin{corollary}\label{Fried-cor} The function on the rational points of a fibered face $F$ that sends $\overline \alpha$ to $L(S_\alpha,\phi_\alpha)$
extends to a real analytic, strictly convex function on $F$ that goes to infinity toward the boundary of $F$.
\end{corollary}

\begin{proof}
By homogeneity,  we have 
$$
 \log(L(S_\alpha,\phi_\alpha)) = ||\alpha|| \log (\lambda (\phi_\alpha)) = \mathcal H (\overline \alpha).
$$
\end{proof}

\begin{remark}{\em Strict convexity of $\mathcal H$ and its behavior toward the boundary of $F$ imply that this
function has a unique minimum on $F$. The minimum, however,
does not necessarily occur at a rational point, and hence it may not be realized by the monodromy of
a circle fibration \cite{Sun12}.
}
\end{remark}

\begin{corollary} Any convergent sequence on the interior of a fibered
face that is not eventually constant corresponds to a family of pseudo-Anosov mapping classes with unbounded Euler
characteristic and bounded normalized dilatation. 
\end{corollary}

Farb, Leininger and Margalit prove the following partial converse.

\begin{theorem}[Universal Finiteness Theorem \cite{FLM09}]\label{FLM-thm}  Let $\Phi$ be a family of pseudo-Anosov mapping classes
with the property that for some constant $C > 1$,
we have 
$$
L(S,\phi) < C
$$
for all $(S,\phi)$ in $\mathcal F$.
Then there is a finite
set of manifolds $\mathcal M = \{M_1,\dots,M_k\}$ so that the mapping torus $(S^0,\phi^0)$
corresponding to each element of $\Phi$ is an element of $\mathcal M$.
\end{theorem}

\noindent
It follows that
to study the dynamics of a family of mapping classes with bounded normalized dilatation, it suffices to look
at a finite collection of fibered faces of hyperbolic 3-manifolds.

\subsection{Teichm\"uller polynomials}
In \cite{McMullen:Poly}, McMullen defined, for each fibered hyperbolic 3-manifold $M$, and fibered face $F \subset H^1(M;\R)$,
an element  $\Theta_F \in \Z G$, called the {\it Teichm\"uller polynomial} 
where $\Z G$ is the group ring over $G = H_1(M;\Z)/{\mbox{torsion}}$.  Since $G$ is a free
abelian group, we can identify elements with monomials in the generators of $G$, and think of elements of $\Z G$ as
polynomials in several variables with integer coefficients.  Given an element $\theta \in \Z G$, written
$$
\theta = \sum_{g \in G} a_g g,
$$
and $\alpha \in H^1(M;\Z)$,
the {\it specialization} of $\theta$ at $\alpha$ is defined by
$$
\theta^{(\alpha)} (t) = \sum_{g \in G} a_g t^{\alpha(g)}.
$$

\begin{theorem}[McMullen \cite{McMullen:Poly}]\label{McMullen-thm} Let $F$ be the fibered face of a hyperbolic 3-manifold.  Then for each
integral $\alpha \in F \cdot \R^+$, the dilatation of $(S_\alpha,\phi_\alpha)$ equals the house of the specialization
$$
\lambda(\phi_\alpha) = |\Theta_F^{(\alpha)}|.
$$
\end{theorem}

Combining the Universal Finiteness Theorem (Theorem~\ref{FLM-thm}) with 
 Penner's result on the asymptotic behavior of minimum dilatations given in Equation (\ref{Penner-eqn}),  it follows that 
there are a finite number of fibered faces that contain points corresponding to mapping classes whose closures 
(obtained by filling in punctures) give rise to mapping classes $(S_g,\phi_g)$ realizing $\lambda(\phi_g) = \delta_g$.
Theorem~\ref{McMullen-thm} shows further that there is a finite set of group ring elements  $\Theta_i \in \Z G_i$, $i=1,\dots,k$,
so that the dilatations of these maps equal the house of specializations of these elements.

We now change notation, and think of group rings $\Z G$ as Laurent polynomial rings.  That is, if $G$ has generators
$t_1,\dots,t_k$, then
there is a natural isomorphism of $\Z G$ with the Laurent polynomial ring $\Lambda(t_1,\dots,t_k) = \Z[t_1^{\pm 1},\dots,t_k^{\pm 1}]$,
where each element of $G$ is considered as a monomial in $t_1,\dots,t_k$.   Similarly, there is an isomorphism of
$\Z^k$ with $\mbox{Hom}(G ;\Z)$  where ${\bf m} = (m_1,\dots,m_k)$ corresponds to the map that sends $t_i$ to $t^{m_i}$, where
we think of $t$ as the generator of $\Z$.  By these identifications, the specialization of $p(t_1,\dots,t_k) \in \Lambda(t_1,\dots,t_k)$, 
at   $\bf m$ is defined by
$$
p^{(\bf m)} (t) = p(t^{m_1},\dots,t^{m_k}).
$$

Putting the Universal Finiteness Theorem (Theorem~\ref{FLM-thm}) together with Theorem~\ref{McMullen-thm}, we have the following.

\begin{theorem}[Universal Finiteness Theorem II] For any constant $C$, there is a finite list of  Laurent polynomials
$p_1,\dots,p_r \in \Z[[t_1,\dots,t_k]]$ so that if $(S,\phi)$ satisfies
$L(S,\phi) < C$, then
$$
\lambda(\phi) = |p_i^{(\bf m)}(t)|
$$
for some $i=1,\dots,r$ and $\bf m \in \Z^k$.
\end{theorem}

\subsection{The magic manifold}\label{magic-sec}
All of the known minimum dilatation examples for punctured as well as closed surfaces are associated, after possibly
adding or removing punctures, to points on the fibered face of the magic manifold
(see \cite{KT:magicmanifold} \cite{KKT:magicmanifold}).  This is the 3-cusped 
hyperbolic 3-manifold that is topologically equal to the complement of the link drawn in Figure~\ref{magic-fig}
in the 3-sphere $S^3$.  The name {\it magic manifold} appears also in the context of hyperbolic 3-manifolds which
admit many non-hyperbolic Dehn fillings, and is the 3-cusped hyperbolic 3-manifold with smallest volume \cite{Gordon:MagicManifold}.

\begin{figure}[htbp]
\begin{center}
\includegraphics[height=1in]{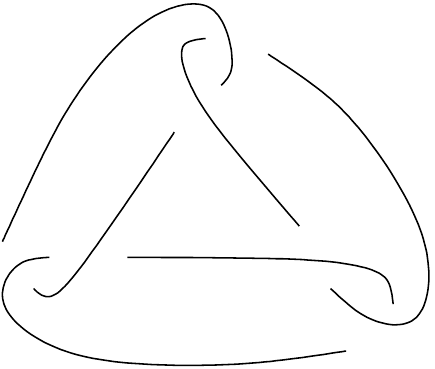}
\caption{Magic Manifold as complement of links in $S^3$.}
\label{magic-fig}
\end{center}
\end{figure}

The first homology group $G = H_1(M;\Z)$ is a free group on 3 generators $x,y,z$ corresponding to 
meridian loops around the component of the link.  The symmetry of the link induces a symmetry on
the Thurston norm.   Let $\hat x, \hat y, \hat z$ be the dual elements.
These form a basis for $H^1(M;\R)$, and $x,y,z$ define coordinate functions on $H^1(M;\R)$.  
With respect to these coordinates, Thurston norm ball is the convex polytope with vertices 
$(\pm 1,0,0), (0,\pm 1,0), (0,0, \pm 1), (\pm 1,\pm 1,\pm 1)$.   Consider the face $F$ defined by
the convex hull of $(1,0,0), (1,1,1), (0,1,0), (0,0,-1)$.   The cone over $F$ can be characterized by
the property 
$$
x + y - z > \max\{x,y,x-z,y-z,0\},
$$
and $F$ is given by
$$
\{(x,y,z)\ : \ x + y - z = 1, x > 0, y > 0, x > z, y> z\}.
$$
We switch to multiplicative notation by replacing $x,y,z$ with $t^x,t^y,t^z$.
Then, the  Teichm\"uller polynomial for $F$ is given by
\begin{eqnarray}\label{MTeich-eqn}
P(t^x,t^y,t^z) &=& t^{x + y - z} - t^x - t^y - t^{x-z} - t^{y-z} + 1.
\end{eqnarray}

\subsection{Dehn Fillings}  Let $M$ be a hyperbolic 3-manifold with cusps.  Each cusp looks topologically like 
$$
S^1 \times S^1 \times (0,\infty),
$$
and we can think of $M$ as the interior of a 3-manifold $M^u$ with torus boundary components. 
A {\it Dehn filling} of $M^u$ at a torus boundary component is the 3-manifold given by attaching a solid torus 
by identifying boundaries.  The filled 3-manifold is determined up to homeomorphism type by the image of the contracting
loop on the surface of the solid torus in $\pi_1(M)$.   This can be specified by a slope
when $M$ is a knot or link complement in $S^3$ as follows. The meridian $\mu$ is the element of
the fundamental group of the torus boundary component that contracts in $S^3$, and the longitude $\gamma$ is
the element whose linking number with the knot in $S^3$ equals zero.  Then  Dehn fillings are
determined by rational numbers $\frac{p}{q}$, where $q \mu + p \gamma$ is the contracting loop.  If the component of the link
is clear, we write the Dehn filling as $M(\frac{b}{a})$.  Thus, for example, if $M$ is the complement of a knot in $S^3$, then
$M(0) = S^3$.  If $M'$ is obtained from the complement $M$ of a link with $k$ components $\ell_1,\dots,\ell_k$ with
meridians $\mu_i$ and longitudes $\gamma_i$, then we write $M'$ as $M' = M(\frac{p_1}{q_1};\dots;\frac{p_k}{q_k})$.

If $M$ has a circle fibration $\psi : M \rightarrow S^1$ with monodromy $(S,\phi)$, then the intersection of $S$ with a cusp
of $M$ determines a Dehn filling $M'$ of $M$ along the cusp.  Let $F$ be the fibered face of $M$ containing the dual element $\alpha_S$
of $S$.  The map $H^1(M';\R) \rightarrow  H^1(M;\R)$ defined by the inclusion $M \hookrightarrow M'$ is one-to-one, since every
loop on $M'$ can be pushed off into $M$.  Let $F'$ be the preimage of $F$  in $H^1(M';\R)$.   Since the map
$H_1(M;\R) \rightarrow H_1(M';\R)$  has kernel generated by the contracting loop of the Dehn filling, we have the following.

\begin{proposition}  If the boundary slope is a finite order element of $H^1(M;\R)$, then the inclusion $F' \hookrightarrow F$
is a bijection.  Otherwise, $F'$ maps to a co-dimension one linear section of $F$.
\end{proposition}

The elements of $F'$ inherit many of  the properties of $F$.

\begin{proposition}\label{F'-prop}  Let $\alpha'$ be a rational element of $F'$, and $\alpha$ its image in $F$.
\begin{enumerate}
\item The boundary slopes defined by the intersection of the dual surface $S_\alpha$ with the cusp are all homologically 
equivalent to that defined by $S$.
\item The intersections $S_\alpha'$ with the filled cusp define  a
 periodic orbit of  $\phi'_\alpha$.  
\item  If the points in the periodic orbit do not come from poles of the quadratic differential on $S$ determined (up to scalar multiple)
by the stable and unstable foliations associated to $\phi_\alpha$,  then  $(S_\alpha,\phi_\alpha)$ is pseudo-Anosov and
$$
\lambda(\phi'_\alpha) = \lambda(\phi_\alpha).
$$
\end{enumerate}
\end{proposition}

The proof of parts (1) and (2) of Proposition~\ref{F'-prop} is an easy consequence of the definitions.  Part (3) follows from
the fact that the stable and unstable foliations of $(S_\alpha,\phi_\alpha)$ also form stable and unstable foliations for 
$(S_\alpha',\phi_\alpha')$ as long as the periodic orbit does not consist of poles.   

\begin{remark}{\em
In the case of poles, it is possible that $(S_\alpha',\phi_\alpha')$ is not pseudo-Anosov. In this case,
by Theorem~\ref{hyperbolicity-thm}, it follows that the Dehn filling $M'$ is not hyperbolic, and hence
$(S_\alpha',\phi_\alpha')$ is not pseudo-Anosov for all rational $\alpha' \in F'$.   Such a Dehn filling is called
an {\it exceptional Dehn filling}, and it was shown by Thurston that there are only a finite number of boundary slopes with this
property.  }
\end{remark}

Let $\Theta \in \Z G$ be the Teichm\"uller polynomial for $F$ and $\Theta' \in \Z G'$ the Teichm\"uller polynomial for $F'$,
where $G = H_1(M;\Z)/{\mbox{torsion}}$ and $G' = H_1(M';\Z)/{\mbox{torsion}}$.

\begin{proposition}  If no periodic orbit contains poles, then the Teichm\"uller polynomial of $F'$ is 
a factor of
the specialization of the Teichm\"uller polynomial for $F$ defined by the map $i_* : G \rightarrow G'$ induced by the
inclusion $i : M \rightarrow M'$, that is, if 
$$
\Theta = \sum_g a_g g,
$$
then $\Theta'$ divides $\sum_g a_g i_*(g)$.
\end{proposition}

\begin{remark}{\em
Assuming the case that the periodic orbit does not consist of poles,
the effect of Dehn filling on normalized dilatation is more complicated than for the dilatation itself.
For example, if $\alpha'$ is a rational element of $F'$ and $\alpha$ is its image
in $F$, then 
$$
\chi(S_\alpha) = \chi(S_\alpha') - s_\alpha,
$$
 where $s_\alpha$ is the number of components in the intersection of $S_\alpha$ with the cusp,
and depends on $\alpha$.  Thus, the normalized dilatation function $L$ on $F'$ is not the pull back of the normalized dilatation function
on $F$, and the effect of pull back on the minimizer of normalized dilatation is not obvious.}
\end{remark}

\subsection{Fibered faces of the manifold $M_{\mbox{m}}(\frac{1}{-2})$.}
The minimum dilatation orientable pseudo-Anosov mapping class of genus 8 is the monodromy of  a fibration of $M_s = M_{\mbox{m}}(\frac{1}{-2})$ (see \cite{Hironaka:LT}).  The manifold $M_s$ is homeomorphic
to the complement of the encircled closure of the braid $\sigma_1 \sigma_2^{-1}$, where $\sigma_1$ and $\sigma_2$
are the standard braid generators of the braid group on 3-strands.   This two component link, known as 
$6{}_2^2$ in Rolfsen's knot table \cite{Rolfsen76}, is symmetric in the
two components and can be drawn in two ways (see Figure~\ref{six22link-fig}).

\begin{figure}[htbp]
\begin{center}
\includegraphics[height=1.2in]{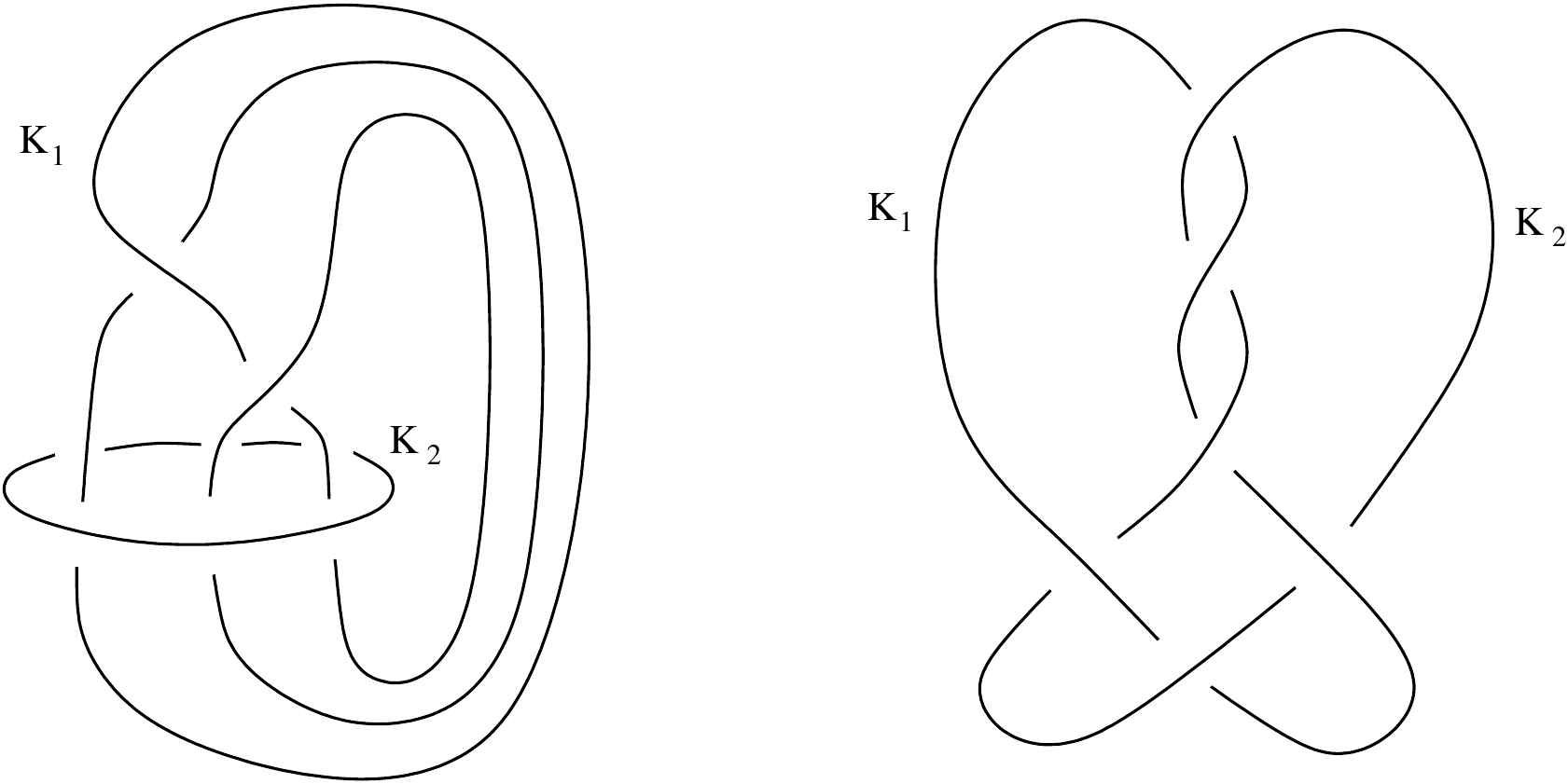}
\caption{Two drawings of the $6{}_2^2$ link.}
\label{six22link-fig}
\end{center}
\end{figure}

Let $M_{\mbox{m}}$ be the magic manifold described in Section~\ref{magic-sec}.
Assume that the Dehn filling is done on the cusp of $M_{\mbox{m}}$ corresponding to the coordinate function $y$.
Then inclusion map $M_{\mbox{m}} \rightarrow M_s$ induces the surjection
$$H_1(M_{\mbox{m}};\R) \rightarrow H_1(M_s;\R)$$
has kernel generated by $t^{y + 2(x +z)}$.   Substituting $x = b$, $z = a$ and $y = -2(b + a)$ in Equation~\ref{MTeich-eqn} gives
\begin{eqnarray*}
P(t^a,t^b) &=& t^{3b + a} - t^{2b + 2a} - t^b - t^ {b-a} - t^{a + 2b} + 1\\
&=& (t^{b + a} + 1)(t^{2b} - t^{b+ a} - t^b - t^{b -a}  + 1).
\end{eqnarray*}

Let $F_{\mbox{m}}$ be the fibered face described in Section~\ref{magic-sec}.
  In \cite{Hironaka:LT}, we show that the fibered face $F_s$ of $M_s$
  corresponding to $F_{\mbox{m}}$ is the locus
  $$
  F_s= \{(x,z) \ : \  x = 1,\  - 1< z< 1\},
  $$
  and the Teichm\"uller polynomial equals
  $$
  \theta_s(t^a,t^b) = t^{2b} - t^{b + a} - t^b - t^{b -a}  + 1.
  $$
 The Alexander polynomial of $M_s$ equals \cite{Rolfsen76}
  $$
  \Delta_s(t^a,t^b) = t^{2b} - t^{b+a} + t^b - t^{b-a} + 1.
  $$
Let $\alpha(a,b)$ denote the element of $H^1(M:\R$ that sends $x$ to $b$ and $z$ to $a$.
If $b$ is even, and $a$ is odd, then
  $$
  |\theta_s(t^a,t^b)| = |\Delta_s(t^a,t^b)|
  $$
and we have the following.
  
  \begin{proposition} \label{orientable-prop} On the fibered face  $F_s$ of $M_s$, the monodromy of 
  $\alpha(a,b)$ is orientable if and only if
  $b$ is even and $a$ is odd, and in particular, it is orientable when $b$ is even and $a = 1$.
  \end{proposition}
  
The monodromy $(S_{(a,b)},\phi_{(a,b)})$ associated to a rational point on $F_s$ whose
primitive element has coordinates $(a,b)$ has topological Euler characteristic equal to minus the degree
of the Alexander polynomial.  Thus, the genus of $S_{(a,b)}$ is given by
$$
g(a,b) = 1 + b - \frac {s}{2},
$$
where $s$ is the number of punctures of $S_{(a,b)}$.  

Let $K_1$ and $K_2$ be the connected components of the $6{}_2^2$-link, and let $\mu_i$ and $\gamma_i$ be their
meridian and longitude for $i=1,2$.  Then $\mu_1$ and $\mu_2$ generate $H_1(M_s;\Z)$ and
$$
\gamma_1 = 3\mu_2 \qquad \gamma_2 = 3\mu_1.
$$
Take any integral $(a,b) \in F_s \cdot \R^+$, and let $\alpha = \alpha(a,b)$.  Let $B_i$ be the boundary tori 
of tubular neighborhoods of $K_i$ in $M_s$.   For $i=1,2$, let $m_i = \alpha(\mu_i)$ and $\ell_i = \alpha(\gamma_i)$ 
be the images of the meridians and longitudes of $K_i$.  Let 
$$
d_1 = \gcd (a,3b) \qquad \mbox{and} \qquad d_2 = \gcd(3a,b).
$$
 Then $d_i$ is the index of the image of $\pi_1(B_i)$ in $\Z$, and hence is equal to the number of
connected components of $S_{(a,b)} \cap B_i$.

In the particular case where $(a,b) = (1,n)$, we have the following.

\begin{lemma} The number of punctures $s$ of $S_{(1,n)}$ is given by
$$
s = \left \{
\begin{array}{ll}
2 &\qquad\mbox{if $3$ doesn't divide $n$}\\
4 &\qquad\mbox{if $3$ divides $n$}
\end{array}
\right .
$$
\end{lemma}
 
 \begin{corollary} The monodromies $(S_{1,g},\phi_{1,g})$, where $g = 2,4 \pmod 6$, have the property that
 \begin{enumerate}
 \item $S_{1,g}$ has genus $g$;
 \item $S_{1,g}$ has two singularities of degrees $3g-2$ and $g-2$, respectively;
 \item $(S_{1,g},\phi_{1,g})$ is orientable; and
 \item $\lambda(\phi_{1,g}) = |LT_{1,g}|$.
\end{enumerate}                                                                                                                                                                                                                                                                                                                                                                                                                                                                                                                                                                                                                                                                                                                                                                                                                                                                                                                                                                                                                                                                                                                                                                                                                                                                                                                                                                                                                                                                                                                                                                                                                                                                                                                                                                                                                                                                                                                                                                                                                                                                                                                                                                                                                                                                                                                                                                                                                                                                                                                                                                                                                                                                                                                                                                                                                                                                                                                                                                                                                                                                                                                                                                                                                                                                                                                                                                                                                                                                                                                                                                                                                                                                                                                                                                                                                                                                                                                                                                                                                                                                                                                                                                                                                                                                                                                                                                                                                                                                                                                                                                                                                                                                                                                                                                                                                                                                                                                                                                                                                                                                                                                                                                                                                                                                                                                                                                                                                                                                                                                                                                                                                                                                                                                                                                                                                                                                                                                                                                                                                                                                                                                                                                                                                                                                                                                                                                                                                                                                                                                                                                                                                                                                                                                                                                                                                                                                                                                                                                                                                                                                                                                                                                                                                                                                                                                                                                                                                                                                                          
 \end{corollary}
 
 By Fried's theorem (Theorem~\ref{Fried-thm}), the function $L(S,\phi)$ extends to a continuous convex function on $F$ that goes to
 infinity toward the boundary.  Thus, it has a unique minimum in $F_s$.  The Teichm\"uller polynomial is invariant under the involution
 on $H_1(M_s;\R)$ given by sending $z$ to $-z$.  It follows that $\lambda(S,\phi))$ is symmetric around the $z=0$ axis,
 and the minimum of $L$ on $F$ occurs at the rational point $\frac{\alpha(0,1)}{||\alpha(0,1)}$, and is given by
 $$
 \lambda(\phi_{(0,1)}) = |t^3 - 3t + 1| = \frac{3 + \sqrt{5}}{2} = \gamma_0^2.
 $$                                                                                                                                                                                                                                                                                                                                                                                                                                                                                                       
 Thus the conjectural minimum accumulation point for genus normalized dilatations of pseudo-Anosov mapping classes 
 (Conjecture~\ref{goldenmean0-conj}).  
 
 Concretely $(S_{(0,1)},\phi_{(0,1)})$ is the mapping class known as the {\it simplest hyperbolic braid}.  Using the left diagram in
 Figure~\ref{six22link-fig}, consider the three times punctured disk $D$ bounded by the encircling link $K_2$.  Then $D$ is 
 Poincare dual to  $\mu_2$ considered as an element of  $H_1(M_s;\Z)$, and hence is the dual surface to $\alpha(0,1)$.  
 The mondromy is defined by considering $M_s$ as the complement of the braid defined by $K_1$ in a solid torus
 given by the complement of a thickened $K_2$ in $S^3$.  The solid torus fibers uniquely up to isotopy over $S^1$ with
 fiber $D$, and the monodromy is the braid monodromy defined by $K_2$, namely the one defined by $\sigma_1\sigma_2^{-1}$,
 where $\sigma_1$ and $\sigma_2$ are the braid generators.

 The points $\alpha(1,n)$ in $H^1(M_s;\R)$ define rays converging to the ray through $\alpha(0,1)$, and hence
 the sequence $L(S_{(1,n)},\phi_{(1,n)})$ converges to $\rightarrow L(S_{(0,1)},\phi_{(0,1)})$.  Since $\chi(D) = -2$, we have
 $$
 \lambda(\phi_{(1,g)})^{2g} = L(S_{(1,g)},\phi_{(1,g)}) \rightarrow L(S_{(0,1)},\phi_{(0,1)}) = \gamma_0^4.
 $$
 This leads to the more general version of Conjecture~\ref{goldenmean0-conj}.
 
 \begin{conjecture}\label{goldenmean2-conj} The smallest accumulation point for normalized dilatations 
 is $\gamma_0^4$.
 \end{conjecture}

 The minimum dilatation orientable pseudo-Anosov mapping classes of genus 7 was found independently in
  \cite{AD10} and \cite{KT11} and is the monodromy of $M_w= M_m(\frac{3}{-2})$, which is  the complement
of the $(-2,3,8)$-pretzel link, also known as the Whitehead sister-link in $S^3$.
The minimum dilatations of pseudo-Anosov mapping classes arising as monodromies of circle fibrations of $M_w$
are all of the form
$|LT_{a,b}|$, where $a \in \{3,7,13,17\}$ and $b=g+2$. 
Putting together the examples above, we have the following.

\begin{proposition} For all $g$
$$
\delta_g \leq |LT_{1,g}|, 
$$
and hence
$$
\limsup (\delta_g)^g \leq \gamma_0^2
$$
and 
$$
\limsup L(S,\phi) \leq \gamma_0^4.
$$
\end{proposition}

Let $\lambda_{(a,b)}  = |LT_{(a,b)}|$, and let $\lambda_{(x,y,z)} = |P(t^x,t^y,t^z)|$.
 In Table~\ref{smalldil-table}, we show the smallest known
dilatations for orientable and unconstrained pseudo-Anosov mapping classes on closed
surfaces of genus 2 through 12.   These put together
 the results in \cite{AD10} (Table 1.9), \cite{KT11} (Thm 1.6, 1.7, 1.12, and Prop. 4.3.7), \cite{KKT:magicmanifold} (Table 1) and \cite{Hironaka:LT} (Prop 4.7).
 
 \begin{table}
\begin{center}
\begin{tabular}{| c || c | c |}
\hline
$g$ & orientable  &  unconstrained \\
\hline
2 &$\lambda_{ (1,2)} \approx 1.72208$ &  same \\
\hline
3 & $\lambda_{(3,4)} \approx 1.40127$ & same \\
\hline
4 & $\lambda_{(1,4)} \approx 1.28064 $& $\lambda_{(3,5)} \approx 1.26123$\\
\hline
5 & $\lambda_{(1,6)} \approx  1.17628 $& $\lambda_{(1,7)} \approx 1.14879$ \\
\hline
6 & $\lambda_{(10,8,3)} \approx 1.20189$ & $\lambda_{(1,8)} \approx 1.12876$\\
\hline
7 & $\lambda_{(2,9)} \approx 1.11548$ & same\\
\hline
8 & $\lambda_{(1,8)} \approx 1.12876$ & $\lambda_{(18,17,7)} \approx 1.10403$\\ 
\hline
9 & $\lambda_{(2,11)} \approx 1.09282 $& same \\
\hline
10 & $\lambda_{(1,10)} \approx 1.10149$ &$ \lambda_{(1,12)} \approx 1.08377$\\
\hline
11 & $\lambda_{(1,12)} \approx 1.08377$ & $\lambda_{(1,13)} \approx  1.07705$\\
\hline
12 & $\lambda_{(12,20,3)} \approx 1.10240$ & $\lambda_{(3,14)} \approx 1.07266$\\
\hline
\end{tabular}
\bigskip

\end{center}
\caption{Smallest known dilatations for genus $g  \leq 12$.}
\label{smalldil-table}
\end{table}

\subsection{Dilatations of pseudo-Anosov mapping classes}  We are particularly interested in the subclass
of pseudo-Anosov mapping classes whose stable and unstable foliations are orientable.
This is equivalent to the condition that
the {\it homological dilatation} $\lambda_{\mbox{hom}}(\phi)$,
which is the spectral radius of the action of $\phi$ on the first homology of $S$, is equal
to the geometric dilatation $\lambda(\phi)$.  Such mapping classes are called
{\it orientable}.  Let $\delta_g^+$ be the minimum dilatation for
orientable pseudo-Anosov mapping classes on $S_g$.  By the results in \cite{Penner91}
and \cite{HK:braidbounds}, $\delta_g^+$ has the same asymptotic behavior as $\delta_g$:
$$
\log (\delta_g^+) \asymp \frac{1}{g}.
$$

In the orientable case,  $\delta_g^+$ has been computed for $g = 2,3,4,5,7,8$ beginning with work by Lanneau and
Thiffeault in \cite{LT09} and continuing with 
\cite{Hironaka:LT}, \cite{AD10} \cite{KT11}.
In \cite{LT09} Lanneau and Thiffeault  also gave the first attempt to  describe the behavior of minimum dilatation explicitly 
as a function of $g$.
Given a polynomial $p(t)$, the {\it house} of $p(t)$ is given by
$$
|p| = \max \{|\mu| \ : \ p(\mu) = 0\}.
$$

\begin{question}\label{LTquestion} Let 
$$
p_n(t) = t^{2n} - t^{n+1} - t^n - t^{n-1} + 1.
$$
Then for even genus $g \ge 2$, 
$$
\delta_g^+ = |p_g|.
$$
\end{question}

If the answer to Question~\ref{LTquestion} is affirmative, then
$$
\liminf_{g \rightarrow \infty} (\delta_g^+)^g \leq \gamma_0^2,
$$
where $\gamma_0$ is the golden mean. This suggests the following conjecture (cf. Conjecture~\ref{goldenmean0-conj}).

\begin{conjecture} \label{goldenmean-conj} The genus-normalized minimum dilatations satisfy
$$
\liminf_{g \rightarrow \infty} (\delta_g^+)^g = \gamma_0^2.
$$
\end{conjecture}

\section{Digraphs and Perron units}\label{Perron-section}  
The dynamics of a pseudo-Anosov mapping class $\phi : S \rightarrow S$, in particular, the structure of the stable and
unstable invariant
foliations,  can be captured in terms of an associated directed graph,
via an associated train track map.   
The train track map defines a Perron-Frobenius linear map $T$ that preserves a symplectic bilinear form, and the
 dilatation of the mapping class equals the 
Perron-Frobenius eigenvalue of $T$. It follows that dilatations are Perron units.
The minimum dilatation problem for pseudo-Anosov mapping classes is closely related in spirit to Lehmer's problem
for Mahler measures of monic integer polynomials posed in \cite{Lehmer33}.
 In this section, we review Lehmer's question on the distribution of algebraic integers, and focus on the particular
case of Perron units.  

\subsection{Mahler measure and Lehmer's question}
Given a monic integer polynomial 
$$
p(t) = t^d + a_{d-1} t^{d-1} + \cdots + a_0, \qquad a_i \in \Z
$$
the {\it Mahler measure} is given by
$$
\mathcal M(p) = \prod_{p(\mu) = 0} \max \{1,|\mu|\}.
$$
In \cite{Lehmer33}, Lehmer asks: {\it is there a positive gap between 1 and the next smallest Mahler measure?}

The smallest known Mahler measure greater than one is called {\it Lehmer's number}
$$
\lambda_L \approx 1.17628,
$$
and its minimal polynomial for $\lambda_L$ is
$$
p_L(t) = t^{10} + t^9 - t^7 - t^6 - t^5 - t^4 - t^3 + t + 1.
$$
By a result of Smyth \cite{Smyth70}, the smallest Mahler measure of a non-reciprocal irreducible polynomial is
approximately $\lambda_S = 1.32472$, which is greater than $\lambda_L$.  Thus to solve Lehmer's problem it suffices to look
at reciprocal polynomials.

\subsection{Normalized house}
The {\it house} of a polynomial is given by
$$
|p| = \max\{|\mu| \ : \ p(\mu) = 0\}.
$$
We have the inequalities
\begin{eqnarray}\label{MahlerHouse-eqn}
|p| \leq \mathcal M(p) \leq |p|^d.
\end{eqnarray}
We call $|p|^d$ the {\it normalized house} of $p(t)$.
It is an open question whether there is a positive gap between 1 and the next smallest normalized house.
If the answer is no, it would imply that there are sequences of Mahler measures converging to 1 from above.

Lehmer's polynomial $p_L$ has only one root outside the unit circle, and hence we
have the first inequality in Equation (\ref{MahlerHouse-eqn}),
$$
|p_L| = \mathcal M(p_L).
$$
The second inequality is also sharp (e.g., take $p(t) = t^n - 2$).

\subsection{Perron numbers}

A {\it Perron-Frobenius matrix} $T$ is an  $n \times n$ matrix whose entries are all non-negative real numbers,
and such that for some $k_0$, the entries of $T^k$ are all positive all $k \ge k_0$.  Given a non-negative
matrix $T = [a_{i,j}]$, one can define an associated directed graph, or {\it digraph}, $D$ with $n$ vertices $v_1,\dots,v_n$ and 
$a_{i,j}$ directed edges from $v_i$ to $v_j$.  By this correspondence $T$ is Perron-Frobenius if and only if $D$ is {\it strongly connected},
i.e., there is a directed path between any two vertices, and {\it aperiodic}, the path lengths of cycles have no common divisor
greater than one \cite{Kitchens98}.  By the Perron-Frobenius theorem, if $T$ is Perron-Frobenius, then 
there is a vector $v$ with positive entries such that $Tv = \lambda v$, for some $\lambda > 1$, and
$\lambda$ is completely determined by these properties.   This $\lambda$ is called
the {\it Perron-Frobenius eigenvalue} of $T$, or {\it dilatation} of $D$.

A {\it Perron number} is a real algebraic integer $\lambda > 1$ such that all algebraic conjugates have complex norm strictly
less than $\lambda$.  An algebraic integer is a Perron number if and only if it is the Perron-Frobenius eigenvalue of a matrix.
Pisot and Salem numbers are examples of Perron numbers.  A Pisot number is an algebraic integer greater than one
all of whose other algebraic conjugates lie strictly inside the unit circle.  A Salem number is an algebraic integer greater than one
all of whose other algebraic conjugates lie
on or inside the unit circle with at least one on the unit circle.   The smallest Pisot number is 
the smallest Mahler measure $\lambda_S$ for non-reciprocal polynomials found by Smyth.  It is not known whether there are Salem numbers arbitrarily close to
1 or whether the infimum of all Mahler measures greater than 1 is a
Salem numbers.   The  smallest known Salem number is Lehmer's number $\lambda_L$.

Graph theory provides an answer to the minimum normalized house problem for Perron numbers and their
defining polynomials.   Recalling the correspondence between Perron-Frobenius matrices and digraphs, one 
notes that the smallest dilatation digraph has the form given in Figure~\ref{mindigraph-fig} (see \cite{Penner91}).
The characteristic polynomial of the digraph is
$$
p_n(t) = t^n - t-1,
$$
for $n \ge 4$.  The polynomial is interesting also in the case $n = 2$, since $|p_2| = \gamma_0$ is the golden mean,
and in the case $n = 3$, since $p_3 = x^3 - x - 1$ is the Smyth polynomial defining
$\lambda_S$.
We also have
$$
\lim_{n \rightarrow \infty} |p_n|^n = 2,
$$
where the convergence is from above.

\begin{figure}[htbp] %  figure placement: here, top, bottom, or page
   \centering
   \includegraphics[width=1.2in]{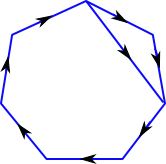} 
   \caption{Minimum dilatation digraph.}
   \label{mindigraph-fig}
\end{figure}

Properties of the normalized house of reciprocal Perron numbers were recently studied in \cite{McM:Clique}, 
showing that any Perron unit  $\alpha$ of degree $n$ satisfies the inequality
$$
\alpha^n \ge \gamma_0^4,
$$
where $\gamma_0$ is the golden mean (see Theorem~\ref{clique-thm}).

\subsection{Complexity of digraphs}
The {\it complexity} $c$ of a digraph
is the number of edges minus the number of vertices of the graph (or minus
the topological Euler characteristic).

\begin{lemma}[Ham-Song \cite{HamSong05}] If $\lambda$ is the spectral radius of $M$,
then $c$ satisfies the inequality
$$
c \leq \lambda^{2n} -1.
$$
\end{lemma}
\begin{figure}[htbp]

\begin{center}
\includegraphics[height=2in]{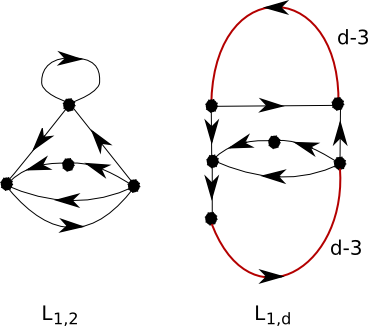}
\caption{Digraphs realizing $LT_{1,n}$.}
\label{digraph-fig}
\end{center}
\end{figure}

Figure~\ref{digraph-fig} shows a family of directed graphs whose characteristic polynomials 
are given by $LT_{1,3}$.   In the Figure, an edge labeled $m$ is subdivided into a chain of $m$ edges
and $m-1$ additional vertices.  Other examples of digraphs with the same dilatation were found in
\cite{Birman:LT}.   The ones shown in Figure~\ref{digraph-fig} have the additional property that they
are defined from the transition matrix of train track maps associated to pseudo-Anosov mapping classes
(see Section~\ref{example-section}).
   
The LT polynomials satisfy
$$
|LT_{1,n}| \leq |LT_{a,n}|
$$
for all $1\leq a<n$, and for any fixed $0<a$,
$$
\lim_{n\rightarrow \infty} |LT_{a,n}|^{2n} = \left (\frac{3 + \sqrt{5}}{2} \right )^2 = \gamma_0^4.
$$
Thus to find the smallest Perron units, it suffices to consider only those with
$\lambda < \lambda_n = |LT_{1,n}|$.  
It follows that to solve the minimum dilatation problem it suffices to look at mapping classes whose
corresponding digraphs have complexity $c \leq 5$.

\subsection{Dilatations of digraphs whose matrices preserve a symplectic form}
It is well-known that any Perron number can be realized as the spectral radius of a Perron Frobenius matrix. 
Furthermore, any Perron unit is the dilatation of a Perron Frobenius matrix that preserves a symplecitc form.
It is not known, however, whether every Perron unit is a dilatation of pseudo-Anosov mapping class.

Given a Perron unit $\lambda$, we define its {\it PF-degree} to be the minimum dimension of a Perron Frobenius matrix
realizing $\lambda$.
McMullen has recently announced the following result giving further support to Conjecture~\ref{goldenmean0-conj}.

\begin{theorem}[McMullen \cite{McM:Clique}]\label{clique-thm} Let $p_{d}$ be the minimum Perron unit of Perron degree 
$d$.  Then
\begin{enumerate}
\item $(p_n)^n \ge \gamma_0^4$ for all $n \ge 1$, and
\item $\lim_{n \rightarrow \infty} (p_n)^n = \gamma_0^4$.
\end{enumerate}
\end{theorem}

\section{Orientable pseudo-Anosov mapping classes} \label{orientable-section}

In \cite{LT09} Lanneau and Thiffeault  studied potential defining polynomials for $\delta_g^+$ 
in the cases $g = 2,\dots,8$, and found lower bounds for $\delta_g^+$ for these $g$.
 Using known examples whose dilatations match these lower bounds they determined $\delta_g^+$ for $g = 2,3,4,5$.   
From the results of Cho and Ham in \cite{CH08}, it follows that $\delta_2 = \delta_2^+$.
Lanneau and Thiffeault's lower bound for $g = 6$ agrees with $\delta_5^+$, showing that $\delta_g^+$ is not
strictly monotone decreasing.  An example realizing $\delta_7^+$ was found in \cite{AD10}
and in \cite{KT11}, and an example realizing  $\delta_8^+$ was found in \cite{Hironaka:LT}.  The exact value for
$\delta_6^+$ is not known.

The minimum dilatations of orientable pseudo-Anosov mapping classes for low genus are given in Table~\ref{mindil-table}.
The associated {\it PF-polynomial} is the  characteristic polynomial of an associated Perron-Frobenius matrix.  This
is not necessarily irreducible.  In  Table~\ref{mindil-table} we repeatedly see the 
cyclotomic factor  $\sigma(t) = t^2 - t + 1$.

\begin{table}[htbp]
\begin{tabular}{|l|l|l|l|}
\hline
g& $\delta_g^+ \approx$ & PF polynomial & factorization\\
\hline
2 & 1.72208 & $t^4-t^3-t^2-t + 1$& irreducible\\
\hline
3 & 1.40127 & $t^{8} - t^7 - t^4 - t+ 1$ & $\sigma(t)(t^6 - t^4 - t^3 - t^2 + 1)$\\
\hline
4 & 1.28064 & $t^8 - t^5-t^4-t^3 + 1$& irreducible\\
\hline
5 & 1.17628 & $t^{12} - t^7 - t^6 - t^5 + 1$ & $\sigma(t)(t^{10} + t^9 - t^7 - t^6 - t^5 - t^4 - t^3 + t + 1)$\\
\hline
7 & 1.11548 & $t^{18} - t^{11} - t^9 - t^7 + 1$ & $\sigma(t)(t^{14} + t^{13} - t^9 - t^8 - t^7 - t^6 - t^5 + t + 1)$\\
\hline
8 & 1.12876 & $t^{16} - t^9 - t^8 - t^7 + 1$& irreducible\\
\hline
\end{tabular}
\bigskip
\caption{List of minimum dilatations and their PF polynomials.}\label{mindil-table}
\end{table}

For $a,b \in \Z$, define the {\it Lanneau-Thiffeault polynomial} $LT_{a,b}$ to be the polynomial
$$
LT_{a,b} (t) = t^{2b} - t^{b+a} - t^b - t^{b-a} + 1.
$$
As can be seen from Table~\ref{mindil-table}, for $g = 2,3,4,5,7,8$, the PF polynomial for the minimum dilatations
of orientable pseudo-Anosov mapping classes is a Lanneau-Thiffeault polynomial.

Question~\ref{LTquestion} can be rephrased as follows.

\begin{question}[Lanneau-Thiffeault Question] For even $g \ge 2$ is it true that
$$
\delta_g^+ = |LT_{1,g}|
$$ 
where $|LT_{1,g}|$ is the house of  $LT_{1,g}(t)$?
\end{question}

By the following result, $|LT_{1,g}|$ is an upper bound for $\delta_g^+$ for $g$ ranging in an arithmetic sequence
or even integers.

\begin{theorem} \label{LTpoly-thm} [\cite{Hironaka:LT}]
For each $g \equiv 2,4 \pmod 6$, there is an orientable pseudo-Anosov mapping class on a genus $g$ closed surface
with dilatation equal to $|LT_{1,g}|$. 
\end{theorem}

\section{Fat train track maps and automata}\label{traintracks-section}

For each pseudo-Anosov mapping class, one can associate a fat
train track map that encodes essential geometric information, including
information about singularities, the invariant stable foliation, and dilatations.
In this section, we give relevant background and definitions.

\subsection{Train tracks and train track maps.}
A train track is a finite topological graph $\tau$ (or 1-complex) with 
no double edges or vertices of degree one.  
A {\it smoothing} of $\tau$ at a vertex $v$ is a choice of tangent
directions for the half edges of $\tau$   that meet at $v$, that is if $e_1$ and $e_2$ meet at a vertex,
then they meet either smoothly or in a cusp.  

 In Figure~\ref{smoothing3-fig}, $e_3$ meets $e_1$ and $e_2$ smoothly, while
$e_1$ and $e_2$ meet at a cusp.

\begin{figure}[htbp]
\begin{center}
\includegraphics[width=3in]{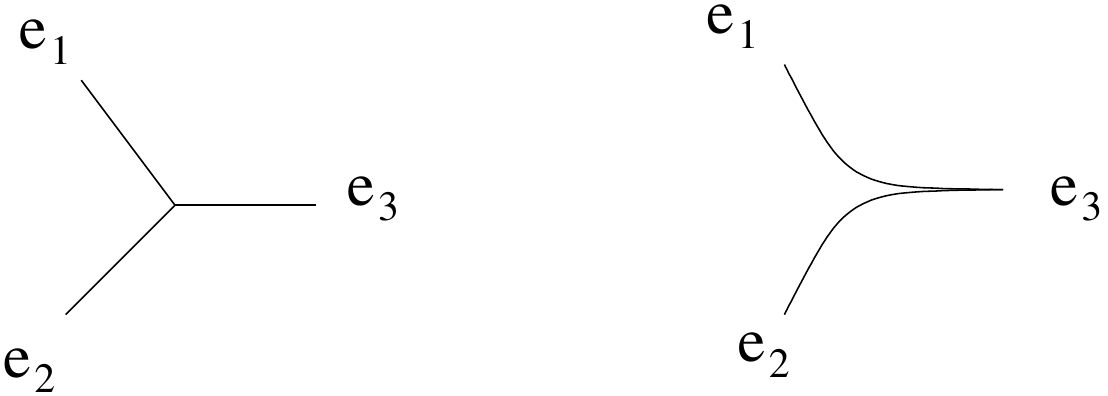}
\caption{Smoothing at a trivalent vertex}
\label{smoothing3-fig}
\end{center}
\end{figure}

\noindent
Figure~\ref{smoothing4-fig} shows a smoothing of a degree four vertex.

\begin{figure}[htbp]
\begin{center}
\includegraphics[width=3in]{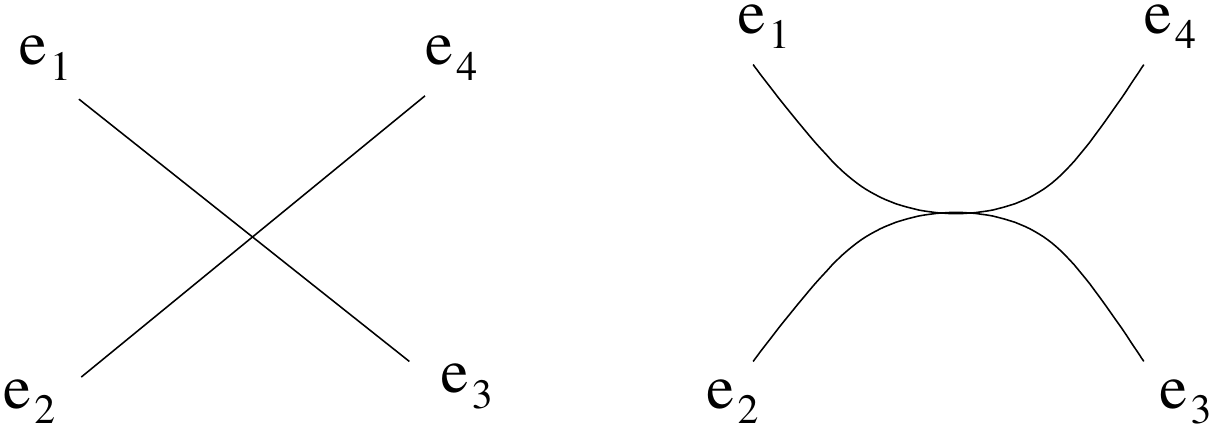}
\caption{Smoothing at a degree 4 vertex}
\label{smoothing4-fig}
\end{center}
\end{figure}

For our examples, we will consider train tracks consisting of a $3b$-gon whose
edges meet in cusps and $2b$-edges attached smoothly to the vertices 
of the $3b$-gon in one of the ways shown in Figure~\ref{smoothing3-fig}
and Figure~\ref{smoothing4-fig}.

By a {\it fat graph}, we mean a graph  such that at any vertex $v$, there is a cyclic ordering of the half edges
that meet at $v$.  This gives a local embedding of the half edges meeting at $v$ into a disk centered at $v$.
Given any fat graph $\Gamma$, there is a canonical 
orientable surface $S_\Gamma$
with boundary  on which $\Gamma$ embeds so that 
\begin{enumerate}
\item at each vertex the ordering of the edges
corresponds to the counterclockwise ordering on the surface; and
\item $S_\Gamma$ deformation retracts to the image of $\Gamma$
under the embedding.
\end{enumerate}
Each boundary component is one boundary component of an annular
complementary component of $\tau$ on $S_\Gamma$.   Consider the
edges surrounding the other {\it interior} boundary component.  Each
time two adjacent edges meet in a cusp, we call it a {\it vertex of the polygon
formed by $\tau$ around the boundary component}.   If the number of
vertices of the polygon is $k$, we say the boundary component is
{\it contained in a $k$-gon} of $\tau$.

A fat train track $\tau$ embedded on a surface $S$  {\it fills}  $S$ if $S$ is obtained from $S_\tau$ by filling in some subset (possibly empty)
of the boundary components of $S_\tau$ with disks.  

A {\it train track map} $f : \tau \rightarrow \tau$ is a  local embedding so that vertices map to vertices, and
edges map to edge-paths on $\tau$ so that no subinterval of an edge passes across two half edges meeting
at a cusp.  We consider train track maps up to isotopy on $\tau$.  

A train track map $f$ determines a linear transformation $\R^{\mathcal{E}}$ to itself as followis.
Let $\mathcal E$ be the set of (unoriented) edges of $\tau$.  Given $e \in \mathcal E$, let
$$
f_*(e) = \sum_{e'} a_{e'} e',
$$
where $a_{e'}$ is the number of times $f(e)$ passes over $e'$.  
Define $T: \R^{\mathcal E} \rightarrow \R^{\mathcal E}$, where for each $w \in \R^{\mathcal E}$, 
$$
T(w)(e) = w(f_*(e)),
$$
where $w$ extends linearly.  The transformation $T$ is called the {\it transition map} defined by $f$.

The {\it weight space} $W_\tau$ of a train track $\tau$ is the subspace of $\R^{\mathcal E}$ consisting of edge labels
so that if three half edges $e_1$, $e_2$ and $e_3$ meet at a vertex as in Figure~\ref{smoothing3-fig}, then
$$
w(e_1) + w(e_2) = w(e_3),
$$
and if $e_1$, $e_2$, $e_3$ and $e_4$ meet as in Figure~\ref{smoothing4-fig}, then
$$
w(e_1) + w(e_2) = w(e_3) + w(e_4).
$$
An edge labeling $w$ determines a labeling on edge paths, which we also
denote by $w$.  Given a train track map $f$ with transition map $T$, we have
$T(W_\tau) = W_\tau$.

A train track $\tau \subset S$ and train track map $f : \tau \rightarrow \tau$ is {\it compatible} with a 
mapping class $(S,\phi)$, if $\tau$ fills $S$ and the induced map $\phi_*$ on $\tau$ equals $f$.

\begin{theorem} If $(S,\phi)$ is pseudo-Anosov, then 
\begin{enumerate}
\item $(S,\phi)$ has a compatible train track $\tau$ and train track map $f : \tau \rightarrow \tau$;
\item the induced map $f_*$ on $W_\tau$ is Perron-Frobenius, and preserves a symplectic form; and
\item $\lambda(\phi)$ is the spectral radius of $f_*$.
\end{enumerate}
\end{theorem}

In the examples that follow, it is possible to find a subcollection of edges
 in  $\mathcal E$  whose duals in $\R^{\mathcal E}$ form a basis
for $W_\tau$.  We call these the {\it real} edges of $\tau$ and the complementary set of edges the {\it infinitessimal} edges.

\subsection{Simplest hyperbolic braid}
Figure~\ref{simplest-fig} gives an example of a fat train track and train track map compatible with the
simplest hyperbolic braid.   The weights in the weight space are determined by their labels on the two longer edges of the
train track, and the three encircling loops are the corresponding infinitessimal edges.   The action of the simplest hyperbolic braid
monodromy defined by $\sigma_1\sigma_2^{-1}$ acts on the real edges according to the matrix
$$
\left [
\begin{array}{ll}
1 & 1\\
1 & 2
\end{array}
\right ],
$$
and the dilatation is the largest eigenvalue $\frac{3 + \sqrt{5}}{2} = \gamma_0^2$.

\begin{figure}[htbp]
\begin{center}
\includegraphics[width=3.5in]{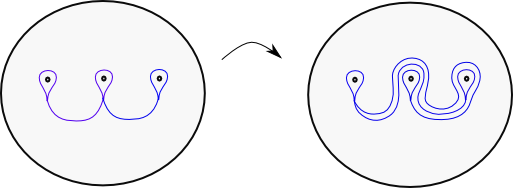}
\caption{Train track for simplest hyperbolic braid monodromy}
\label{simplest-fig}
\end{center}
\end{figure}

\subsection{Orientable train tracks}
Each train track on $S$ determines a foliation on $S$ as follows.
For each complementary region of $\tau$ on $S$ surrounded
by a $k$-gon, the foliation has a $k$-pronged singularity. 
A train track is orientable, if there is an orientation on the edges
so that if two edges meet smoothly at a vertex, the orientations
are compatible.     

Figure~\ref{traintrack-singularity-fig} sketches the foliation around 
a boundary component
of $S$ corresponding to a hexagon on a fat train track.   The
orientation on the train track determines an orientation on the
foliations.

Thus, we have the following.

\begin{proposition} A pseudo-Anosov map $(S,\phi)$ that has a compatible
train track map $f : \tau \rightarrow \tau$, where $\tau$ is orientable,
is orientable.
\end{proposition}

\begin{figure}[htbp] %  figure placement: here, top, bottom, or page
   \centering
   \includegraphics[width=3.5in]{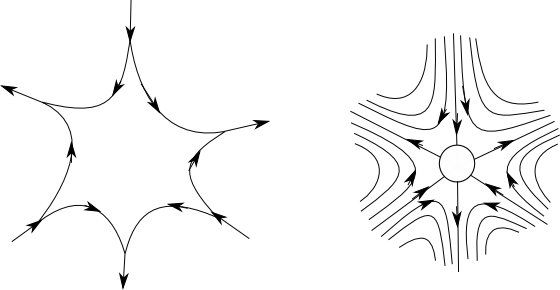} 
   \caption{A hexagon on a fat traintrack, and corresponding foliations.}
   \label{traintrack-singularity-fig}
\end{figure}

\subsection{Train track automaton}
Given two fat train tracks $\tau_1$ and $\tau_2$, a {\it folding map} $\mathfrak f: \tau_1 \rightarrow \tau_2$
is a quotient map obtained by identifying edge-segments of a pair of edge in $\tau_1$ as follows.
Take two edges $e_1$ and $e_2$ on $\tau_1$ with half edges that meet
at a cusp at a vertex $v$, and that are adjacent in the fat graph ordering.  Then the folding
map of $e_1$ over $e_2$ is obtained by identifying the embedded image of a closed interval in $e_1$ with endpoint
$v$ with $e_2$ by a homeomorphism sending $v$ to $v$.
The fat train track automaton is the set of all fat train tracks with a directed edge from one train track to another
if  there is a folding
map between them. 

Each folding map is a homotopy equivalence of graphs and 
defines a linear transformation between edge labels, and between weight spaces.  A circuit in the
fat train track automaton
corresponds to a composition of folding maps together with an homeomorphism of train tracks. 
Thus, the transition matrix for the train track map corresponds to a composition of transition matrices
for folding maps and a permutation matrix.

 A train track automaton is a directed 
graph whose vertices are train tracks and edges are folding maps.

\begin{proposition} [Stallings \cite{Stallings:folds}, Ham-Song \cite{HamSong05}] Any pseudo-Anosov mapping class
can be represented by a circuit on a train track automaton.
\end{proposition}

\section{Small dilatation examples}\label{example-section}
In this section, we define train track maps
for  mapping classes $(S_n,\phi_n)$ for all integers $n \ge 2$, and describe corresponding circuits
in the train track folding automaton, and digraphs.  These train track maps define mapping classes
with the same genus, boundary components, and dilatations as $(S_{1,n},\phi_{1,n})$.

We begin with a fat train track map defining $(S_2,\phi_2)$ 
in Figure~\ref{ttsix-fig} .
One can check that all of the train tracks in the circuit shown in Figure~\ref{ttsix-fig} fix a genus two surface with two complementary
disk components, one bounded by the central hexagon, and the other bounded by the edges of the hexagon and by each side of the four
real edges.  The train track map defined by composing the folded mapping
classes described in the circuit corresponds to the orientable pseudo-Anoosv mapping classes whose dilatation 
realizes $\delta_2 = \delta_2^+$.  

The center hexagon is made up of infinitessimal edges and the other four edges are real edges.  
Starting at the upper left train track in the  the automaton, we first fold edge $a$ over edge $c$
and the following adjacent infinitessimal edge.  In the next step we fold $b$ over the new edge $a$.
Then we fold the new edge $b$ over $c$.  Finally by a rotation, we return to the original train track.

\begin{figure}[htbp]
\begin{center}
\includegraphics[width=3.5in]{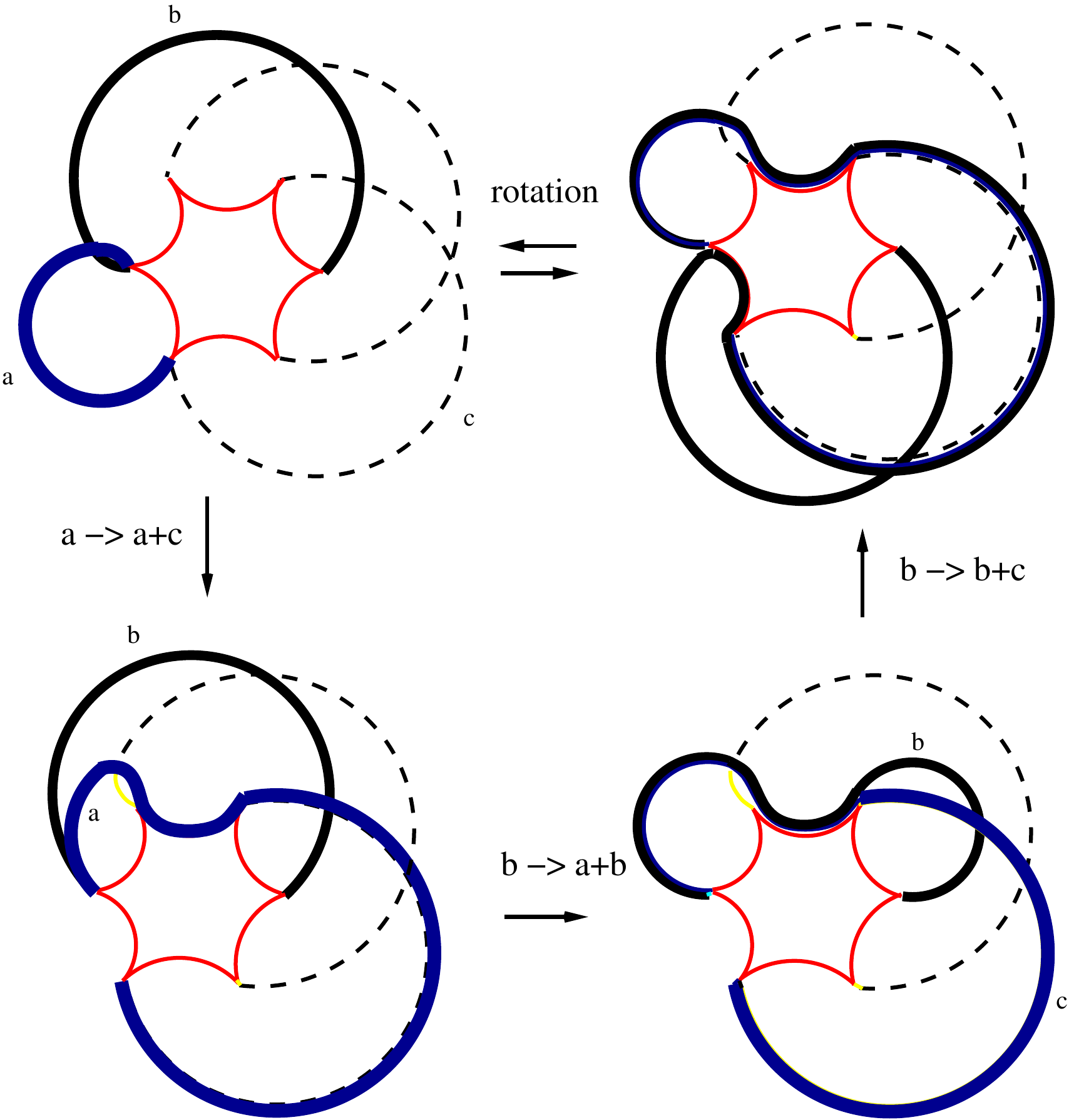}
\caption{Train track circuit for example realizing $\delta_2^+$ and $\delta_2$.}
\label{ttsix-fig}
\end{center}
\end{figure}

The transition matrices for the folding diagrams starting at the top left and going around counter-clockwise are:
\medskip

$$
\left [
\begin{array}{llll}
1 & 0 & 0 & 0\\
0 & 1 & 0 & 0\\
1 & 0 & 1 & 0\\
0 & 0 & 0 & 1
\end{array}
\right ]
\left [
\begin{array}{llll}
1 & 1 & 0 & 0\\
0 & 1 & 0 & 0\\
0 & 0 & 1 & 0\\
0 & 0 & 0 & 1
\end{array}
\right ]
\left [
\begin{array}{llll}
1 & 0 & 0 & 0\\
0 & 1 & 0 & 0\\
0 & 1 & 1 & 0\\
0 & 0 & 0 & 1
\end{array}
\right ]\qquad\mbox{and}\qquad
\left [
\begin{array}{llll}
1 & 0 & 0 & 0\\
0 & 0 & 0 & 1\\
0 & 1 & 0 & 0\\
0 & 0 & 1 & 0
\end{array}
\right ].
$$
The composition is given by
$$
\left [
\begin{array}{llll}
1 & 1 & 0 & 0\\
0 & 0 & 0 & 1\\
0 & 1 & 0 & 0\\
1 & 1 & 1 & 0\\
\end{array}
\right ]
$$
\medskip

\noindent
and its characteristic polynomial is $x^4 - x^3 - x^2 - x + 1$.  This gives
$$
\delta_2 = \delta_2^+ = |x^4 - x^3 - x^2 - x + 1| \approx 1.72208.
$$

The train track in Figure~\ref{ttsix-fig} generalizes to the one in 
Figure~\ref{tt-fig}.

\begin{figure}[htbp]
\begin{center}
\includegraphics[width=4in]{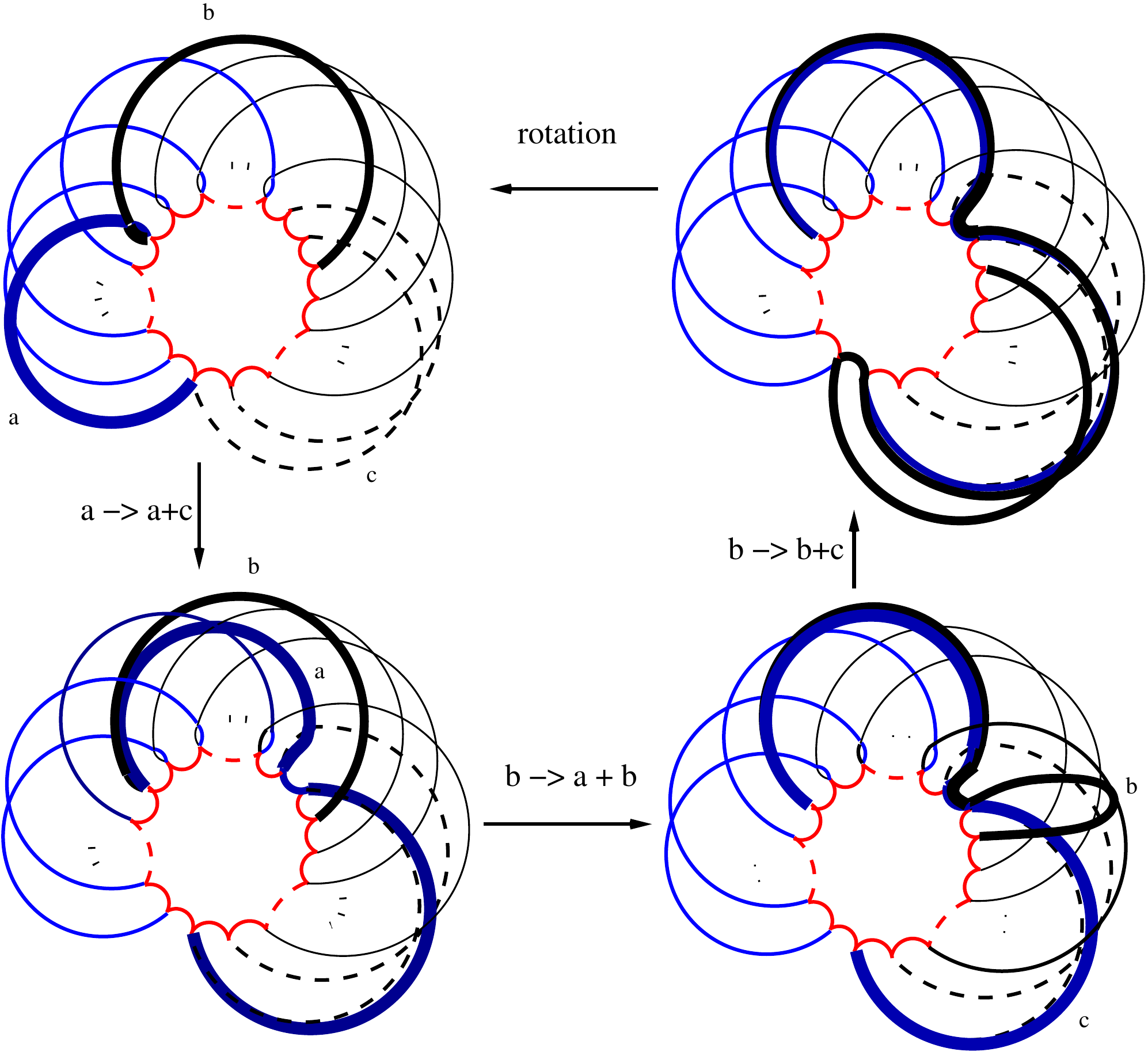}
\caption{Circuit in train track automaton for $(S_n,\phi_n)$}
\label{tt-fig}
\end{center}
\end{figure}

Let $G_n$ be the digraphs in Figure~\ref{digraph-fig}. 
The ``shape" of the train track map and folding maps for $(S_n, \phi_n)$ are related 
to each other in a systematic way, and one observes the following.

\begin{proposition} The digraphs associated to the transition matrices
for the train track maps of $(S_n,\phi_n)$ are $G_n$, and hence
the dilatations of $(S_n,\phi_n)$ are given by
$$
\lambda(\phi_n) = |LT_{1,n}|.
$$
\end{proposition}

The genus of $S_n$ can be determined from the topological Euler characteristic of $G_n$,
$\chi(G_n)=2n$ and the number of boundary components of the fat graph.  There is
one component for the central $3n$-gon, and either one or three other boundary
components, depending on whether $n$ is divisible by 3.  This implies the following.

\begin{proposition} The surface $S_n$ has genus $g=n$ if $n=1,2\pmod 3$,
and has genus $g=n-1$ if $n=0\pmod 3$.
\end{proposition}

From the train track maps, we can also determine when the mapping classes are orientable, for
this is exactly when the train tracks themselves are orientable as seen in the next proposition.

\begin{proposition} The mapping class $(S_n,\phi_n)$ is orientable if and only if $n$ is even.
\end{proposition}

\begin{proof}  The complementary region of $(S_n,\phi_n)$ splits
into a central $3n$-gon and either one $n$-gon, or three  $n/3$-gons, 
depending on whether or not $n$ is divisible by 3.
In order for the train track to be orientable, we need to have each
polygon have an even number of sides.  Thus, $n$ must be even. 

When $n$ is even, there are two possible ways to orient the central $3n$-gon.
Each extends to a compatible orientation on the entire train track.  (An example
is shown in Figure~\ref{orientedtraintrack-fig}).
\end{proof}

\begin{figure}[htbp] %  figure placement: here, top, bottom, or page
   \centering
   \includegraphics[width=2.5in]{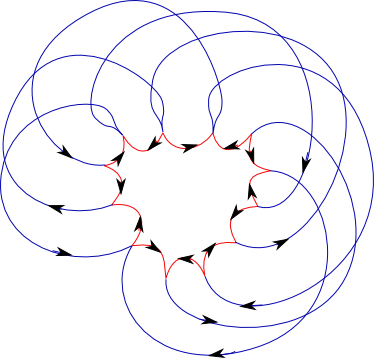} 
   \caption{Oriented train track for $n=4$.}
   \label{orientedtraintrack-fig}
\end{figure}

%The train tracks described in Section~\ref{example-section} have transition matrices  associated  to the digraphs shown in
%Figure~\ref{digraph-fig}. (Here the labeled edges indicate number of subdivisions of the edge).

%A square integer matrix $M$ is {\it Perron-Frobenius}, or {\it PF} if it has the property that
%all entries are non-negative, and all entries of some power are strictly positive.
%A {\it Perron number} is any algebraic integer that can be realized as the spectral radius
%of a Perron-Frobenius matrix.   An {\it (algebraic) unit} is an algebraic integer
%whose reciprocal is also algebraic.  By abuse of notation, we will say $\alpha$
%is a {\it Perron unit} if it and its reciprocal are eigenvalues of the same PF matrix,
%and the {\it Perron degree} is the minimum possible dimension of the PF matrix.
%
%To any PF matrix $M =[a_{i,j}]$ we can associate a directed graph,
%or {\it digraph}, with ordered vertices corresponding to the rows or columns of $M$ and
%$a_{i,j}$ edges from the $i$th to the $j$th vertex.   

\begin{remark} 
{\em In \cite{LT09}, it is shown that the mapping class on a closed genus 2 surface with
minimum dilatation is unique up to standard equivalences.  Thus, the mapping class described in
Figure~\ref{ttsix-fig} is the same as the genus 2 monodromy of the complement of the link drawn in
Figure~\ref{six22link-fig}.  In a forthcoming paper, we will further elaborate on the construction shown 
in this section.}
\end{remark}

\bibliographystyle{../../math}
\bibliography{../../math}

\bigskip

Department of Mathematics,
Florida State University,
1017 Academic Way,
Tallahassee, FL 32306-4510.
email: hironaka$@$math.fsu.edu
 \end{document}